\documentclass[oneside]{amsart}
\usepackage[T1]{fontenc}
\usepackage[latin9]{inputenc}
\usepackage{mathrsfs}
\usepackage{bm}
\usepackage{amsthm}
\usepackage{amstext}
\usepackage{amssymb}
\usepackage[all]{xy}

\makeatletter
\numberwithin{equation}{section}
\numberwithin{figure}{section}
\theoremstyle{plain}
\newtheorem{thm}{\protect\theoremname}[section]
  \theoremstyle{plain}
  \newtheorem{conjecture}[thm]{\protect\conjecturename}
  \theoremstyle{definition}
  \newtheorem{defn}[thm]{\protect\definitionname}
  \theoremstyle{plain}
  \newtheorem{prop}[thm]{\protect\propositionname}
  \theoremstyle{plain}
  \newtheorem{lem}[thm]{\protect\lemmaname}
  \theoremstyle{plain}
  \newtheorem{cor}[thm]{\protect\corollaryname}
  \theoremstyle{remark}
  \newtheorem{rem}[thm]{\protect\remarkname}

\usepackage{stmaryrd}

\makeatother

  \providecommand{\conjecturename}{Conjecture}
  \providecommand{\corollaryname}{Corollary}
  \providecommand{\definitionname}{Definition}
  \providecommand{\lemmaname}{Lemma}
  \providecommand{\propositionname}{Proposition}
  \providecommand{\remarkname}{Remark}
\providecommand{\theoremname}{Theorem}

\begin{document}

\title{Shintani zeta functions and a refinement of Gross's leading term
conjecture}

\author{Minoru Hirose}

\begin{abstract}
We introduce the notion of \emph{Shintani data}, which axiomatizes
algebraic aspects of Shintani zeta functions. We develop the general
theory of Shintani data, and show that the order of vanishing part
of Gross's conjecture follows from the existence of a Shintani datum.
Recently, Dasgupta and Spiess proved the order of vanishing part of
Gross's conjecture under certain conditions. We give an alternative
proof of their result by constructing a certain Shintani datum. We
also propose a refinement of Gross's leading term conjecture by using
the theory of Shintani data. Out conjecture gives a conjectural construction
of localized Rubin-Stark elements which can be regarded as a higher
rank generalization of the conjectural construction of Gross-Stark
units due to Dasgupta and Dasgupta-Spiess. 
\end{abstract}
\maketitle
\tableofcontents{}

\section*{Introduction}

In this paper, we formulate a conjecture concerning to the special
value of Shintani zeta functions. This conjecture has two aspects.
The one is a refinement of Gross's leading term conjecture. The other
is a conjectural construction of localized Rubin-Stark elements.

\subsection{\label{sub:Rubin-Stark-conjecture}Rubin-Stark conjecture and Rubin-Stark
elements}

Let $F$ be a totally real field. For a finite abelian extension $E$
of $F$ and a set $S$ of places of $F$, we write $S_{E}$ for the
set of places of $E$ lying above places in $S$. For a finite place
$v$ of $F$, we write $F_{v}$ for the completion of $F$ at $v$
and $O_{v}$ for the maximal compact subring of $F_{v}$. We write
$S_{\infty}$ for the set of infinite places of $F$.

Let $S$ and $T$ be disjoint finite sets of places of $F$ such that
$S\supset S_{\infty}$. For a finite abelian extension $E$ of $F$,
we write $\mathcal{O}_{E,S}$ for the ring of $S_{E}$-integers of
$E$ and $\mathcal{O}_{E,S,T}^{\times}$ for the group of units of
$\mathcal{O}_{E,S}$ congruent to $1$ modulo all places in $T_{E}$.
For a finite abelian extension $E$ of $F$ unramified outside $S$,
we write $\Theta_{E,S,T}(s)$ for the Stickelberger function. Put
\[
\Theta_{E,S,T}=\Theta_{E,S,T}(0)\in\mathbb{Q}[{\rm Gal}(E/F)].
\]

Let $v_{1},\dots,v_{r}$ be distinct elements of $ $$S$. Put $V=\{v_{1},\dots,v_{r}\}$.
Assume that $V\neq S$. Let $H$ be a finite abelian extension of
$F$ unramified outside $S$ such that $v_{1},\dots,v_{r}$ split
completely in $H/F$. Fix places $w_{1},\dots,w_{r}$ of $H$ lying
above $v_{1},\dots,v_{r}$. In Section \ref{sec:ZGmodules}, we define
a certain submodule $\Lambda_{H,S,T}$ of $(\bigotimes_{\mathbb{Z}}^{r}\mathcal{O}_{H,S,T}^{\times})\otimes_{\mathbb{Z}}\mathbb{Z}[{\rm Gal}(H/F)].$
The module $\Lambda_{H,S,T}$ is isomorphic to the Rubin-Stark lattice
$\Lambda_{S,T}\subset\mathbb{Q}\wedge_{\mathbb{Z}[{\rm Gal}(H/F)]}^{r}\mathcal{O}_{H,S,T}^{\times}$
defined in \cite{MR1385509}. Rubin-Stark conjecture is equivalent
to the following conjecture (see Conjecture \ref{Conj:RubinStark-1}).
\begin{conjecture}
[Rubin-Stark conjecture, equivalent to Conjecture B in \cite{MR1385509}]\label{Conj:RubinStark}Let
$R_{\infty}:(\bigotimes_{\mathbb{Z}}^{r}\mathcal{O}_{H,S,T}^{\times})\otimes_{\mathbb{Z}}\mathbb{Z}[{\rm Gal}(H/F)]\to\mathbb{R}[{\rm Gal}(H/F)]$
be a homomorphism defined by
\[
R_{\infty}(u_{1}\otimes\cdots\otimes u_{r}\otimes[\sigma])=\left(\prod_{j=1}^{r}-\log\left|u_{j}\right|_{w_{j}}\right)[\sigma].
\]
Then there exists a unique element $\epsilon_{H,S,T,V}\in\Lambda_{H,S,T}$
such that
\[
R_{\infty}(\epsilon_{H,S,T,V})=\lim_{s\to0}s^{-r}\Theta_{H,S,T}(s).
\]

\end{conjecture}
The element $\epsilon_{H,S,T,V}\in\Lambda_{H,S,T}$ in Conjecture
\ref{Conj:RubinStark} is called \emph{Rubin-Stark element}.

\subsection{\label{sub:Refinement-of-GLTC}A refinement of Gross's leading term
conjecture}

Let $(H/F,S,T,(v_{i})_{i},(w_{i})_{i},)$ be as in Section \ref{sub:Rubin-Stark-conjecture}.
Put $V=\{v_{1},\dots,v_{r}\}$. Let $K$ be a finite abelian extension
of $F$ unramified outside $S$ such that $\mathcal{O}_{K,S,T}^{\times}$
is torsion free and $H\subset K$. Put $G={\rm Gal}(K/F)$. By a result
of Deligne-Ribet \cite{MR579702}, $\Theta_{K,S,T}$ is in $\mathbb{Z}[G]$.
Let $G_{v}\subset G$ be the decomposition group at a place $v$ of
$F$. For $1\leq j\leq r$, we denote by $f_{j}$ the natural composite
map 
\[
\mathcal{O}_{H,S,T}^{\times}\hookrightarrow H^{\times}\hookrightarrow H_{w_{j}}^{\times}=F_{v_{j}}^{\times}\xrightarrow{{\rm rec}}G_{v_{j}}\hookrightarrow G.
\]
For an abelian group $A$ and a subgroup $U$ of $A$, we write $I_{U,A}$
for the kernel of the projection $\mathbb{Z}[A]\to\mathbb{Z}[A/U]$.
We put
\begin{align*}
I_{H}' & =I_{{\rm Gal}(K/H),G}\subset\mathbb{Z}[G]\\
D' & =\prod_{v\in V}I_{G_{v},G}\subset\mathbb{Z}[G].
\end{align*}
The (usual) regulator 
\[
R_{H/F,S,T,G}:(\bigotimes_{\mathbb{Z}}^{r}\mathcal{O}_{H,S,T}^{\times})\otimes_{\mathbb{Z}}\mathbb{Z}[{\rm Gal}(H/F)]\to D'/I_{H}'D'
\]
 is a homomorphism defined by 
\[
R_{H/F,S,T,G}(u_{1}\otimes\cdots\otimes u_{r}\otimes[\tau])=[\tau]\prod_{j=1}^{r}([f_{j}(u_{j})]-[1]).
\]
The Gross's leading term conjecture is as follows.
\begin{conjecture}
[Gross's leading term conjecture]\label{Conjecute:Gross}We have
\[
\Theta_{K,S,T}\equiv R_{H/F,S,T,G}(\epsilon_{H,S,T,V})\ \pmod{I_{H}'D'}.
\]

\end{conjecture}
The original form of this conjecture was given by Gross \cite{MR931448}
and refined by many authors. This conjecture is one of such a refinement
by the ideas of Tate \cite{MR2088717}, Aoki \cite{MR2144953} and
Burns \cite{MR2336038}, and equivalent to MRS$(K/H/F,S,T,\emptyset,V)$
in \cite{BKS}. Since $R_{H/F,S,T,G}(\epsilon_{H,S,T,V})\in\prod_{v\in V}I_{G_{v},G}$,
this conjecture implies the following conjecture.
\begin{conjecture}
[The order of vanishing part of Gross conjecture]\label{Conj:VanishGross}We
have
\[
\Theta_{K,S,T}\in D'.
\]

\end{conjecture}
Fix an open subgroup $J$ of $\prod_{v\in S}F_{v}^{\times}$ which
can be written as a direct product of open subgroups 
\[
J_{v}\subset F_{v}^{\times}\ \ \ (v\in S).
\]
We denote by $\Upsilon(J)$ the set of finite abelian extensions $E$
of $F$ unramified outside $S$ such that $J\subset\ker(\prod_{v\in S}F_{v}^{\times}\xrightarrow{{\rm rec}}{\rm Gal}(E/F))$.
We assume that $H\in\Upsilon(J)$. For a place $v\notin S$ of $F$,
put $J_{v}=O_{v}^{\times}$. For a place $v$ of $F$, put $N_{v}^{(J)}=F_{v}^{\times}/J_{v}$.
Put 
\[
N_{F}^{(J)}=\mathbb{A}_{F}^{\times}/\prod_{v}J_{v}=\prod_{v}{}^{'}N_{v}^{(J)}
\]
where $\mathbb{A}_{F}^{\times}$ is the idele group of $F$ and $v$
runs all places of $F$. If there is no risk of confusion, we simply
write $N_{F}$ for $N_{F}^{(J)}$ and $N_{v}$ for $N_{v}^{(J)}$.

Put
\begin{align*}
D & =\prod_{v\in V}I_{N_{v},N_{F}}\subset\mathbb{Z}[N_{F}]\\
I_{H} & =\ker(\mathbb{Z}[N_{F}]\xrightarrow{{\rm rec}}\mathbb{Z}[{\rm Gal}(H/F)])\\
I_{F^{\times}} & ={\rm ker}(\mathbb{Z}[N_{F}]\to\mathbb{Z}\otimes_{\mathbb{Z}[F^{\times}]}\mathbb{Z}[N_{F}]).
\end{align*}
For $1\leq j\leq r$, we denote by $\eta_{j}$ the natural composite
map $\mathcal{O}_{H,S,T}^{\times}\hookrightarrow H^{\times}\hookrightarrow H_{w_{j}}^{\times}=F_{v_{j}}^{\times}\to N_{v_{j}}\hookrightarrow N_{F}$.
We define the (enhanced) regulator map 
\[
\hat{R}_{H/F,S,T,J}:(\bigotimes_{\mathbb{Z}}^{r}\mathcal{O}_{H,S,T}^{\times})\otimes_{\mathbb{Z}}\mathbb{Z}[{\rm Gal}(H/F)]\to D/I_{H}D
\]
by
\[
u_{1}\otimes\cdots\otimes u_{r}\otimes[\sigma]\mapsto[\bar{\sigma}]\prod_{j=1}^{r}\left([\eta_{j}(u_{j})]-[1]\right)
\]
where $\bar{\sigma}\in N_{F}$ is any element of ${\rm rec}^{-1}(\sigma)$.
If $K\in\Upsilon(J)$ then we have
\[
R_{H/F,S,T,G}=\varphi\circ\hat{R}_{H/F,S,T,J}
\]
where $\varphi:\mathbb{Z}[N_{F}^{(J)}]\to\mathbb{Z}[G]$ is a homomorphism
induced by the reciprocity map. 

Fix $J=\prod_{v\in V}J_{v}\subset\prod_{v\in V}F_{v}^{\times}$. Assume
that $K\in\Upsilon(J)$. See the diagram
\[
\xymatrix{\mathbb{Z}[N_{F}]/I_{F^{\times}}D\ar^{f_{1}}[r]\ar^{f_{2}}[d] & \mathbb{Z}[G]\ar_{f_{3}}[d]\\
\mathbb{Z}[N_{F}]/I_{H}D\ar_{f_{4}}[r] & \mathbb{Z}[G]/I_{H}'D'
}
\]

where $f_{1},f_{4}$ are maps induced by the reciprocity map and $f_{2},f_{3}$
are projections. Note that $\Theta_{K,S,T}$ is an element of $\mathbb{Z}[G]$
and $\hat{R}_{H/F,S,T,J}(\epsilon_{H,S,T,V})$ is an element of $\mathbb{Z}[N_{F}]/I_{H}D$.
Conjecture \ref{Conjecute:Gross} is equivalent to the statement
\[
f_{3}(\Theta_{K,S,T})=f_{4}(\hat{R}_{H/F,S,T,J}(\epsilon_{H,S,T,V})).
\]
In this paper, under some conditions, we construct a special element
$\hat{\Theta}_{S,T,V,J}$ of $\mathbb{Z}[N_{F}]/I_{F^{\times}}D$
such that $f_{1}(\hat{\Theta}_{S,T,V,J})=\Theta_{K,S,T}$. Furthermore,
we propose a conjecture (see Conjecture \ref{Conj:main} below)
\[
f_{2}(\hat{\Theta}_{S,T,V,J})=\hat{R}_{H/F,S,T,J}(\epsilon_{H,S,T,V}).
\]
In Section \ref{sec:Construction}, we introduce the notion of an
$(S,T,V,J)$-\emph{Shintani datum}. An \emph{$(S,T,V,J)$}-Shintani
datum\emph{ }is a triple $(\mathcal{F},\varphi,x)$ satisfying some
conditions. Let $(\mathcal{F},\varphi,x)$ be an $(S,T,V,J)$-Shintani
datum. In Section \ref{sec:Construction}, we construct an element
$Q(\mathcal{F},\varphi,x)$ of $D/I_{F^{\times}}D$ such that
\[
f_{1}(Q(\mathcal{F},\varphi,x))=\Theta_{K,S,T}
\]
for all $K\in\Upsilon(J)$. This gives a following theorem (see Theorem
\ref{thm:GrossFromShintani}).
\begin{thm}
If an $(S,T,V,J)$-Shntani datum exists then Conjecture \ref{Conj:VanishGross}
is true for $K\in\Upsilon(J)$.
\end{thm}
Thus the order of vanishing part of Gross's conjecture is reduced
to the existence of an $(S,T,V,J)$-Shintani datum. The author hopes that
this axiomatic approach will lead to the full solution of the vanishing
order part of Gross's conjecture in further investigations.

We say that a prime ideal $\eta$ is degree one if $N\eta$ is a rational
prime. Let us consider the following conditions.
\begin{defn}
We say that $(S,T,V)$ satisfies (C1) if\end{defn}
\begin{itemize}
\item $V\not\supset S_{\infty}$,
\item no prime of $S$ has the same residue characteristic as any prime
in $T$,
\item no two primes in $T$ have the same residue characteristic,
\item $T$ contains at least two prime ideal of degree one, or T contains
at least one prime ideal $\eta$ of degree one such that $N\eta\geq n+2$.
\end{itemize}
If (C1) holds then $\Theta_{E,S,T}$ is in $\mathbb{Z}[{\rm Gal}(E/F)]$
for any abelian extension $E/F$ unramified outside $S$. 

In Section \ref{sec:ShintaniDatumBLT}, for all $J$, we construct
a certain Shintani datum $(B,\mathcal{L},\vartheta)$ under the conditions
that $(S,T,V)$ satisfies (C1). This gives an alternative proof of
the following theorem by Dasgupta and Spiess (see Corollary \ref{cor:GrossFromC1}).
\begin{thm}
[\cite{DasguptaSpiess}]If $(S,T,V)$-satisfies (C1) then Conjecture
\ref{Conj:VanishGross} is true.
\end{thm}
Assume that $(S,T,V)$-satisfies (C1). We define an element $\hat{\Theta}_{S,T,V,J}$
of $D/I_{F^{\times}}D$ by
\[
\hat{\Theta}_{S,T,V,J}=Q(B,\mathcal{L},\vartheta).
\]
We propose a following conjecture (see Conjecture \ref{Conj:mainLastSection}).
\begin{conjecture}
\label{Conj:main}For all $J$, 
\[
f_{2}(\hat{\Theta}_{S,T,V,J})=\hat{R}_{H/F,S,T,J}(\epsilon_{H,S,T,V}).
\]

\end{conjecture}
This is a refinement of Conjecture \ref{Conjecute:Gross}.

\subsection{Conjectural construction of localized Rubin-Stark element.}

Let $\epsilon_{H,S,T,V}$ be the Rubin-Stark element. The natural
composite maps
\[
H^{\times}\hookrightarrow H_{w_{j}}^{\times}=F_{v_{j}}^{\times}\to N_{v_{j}}^{(J)}
\]
induce a map
\[
{\rm loc}_{J}:\bigotimes_{j=1}^{r}\mathcal{O}_{H,S,T}^{\times}\otimes_{\mathbb{Z}}\mathbb{Z}[{\rm Gal}(H/F)]\to\bigotimes_{j=1}^{r}N_{v_{j}}^{(J)}\otimes_{\mathbb{Z}}\mathbb{Z}[{\rm Gal}(H/F)].
\]
Put
\[
{\rm loc}=(\lim_{\substack{\leftarrow\\
J
}
}{\rm loc}_{J}):\bigotimes_{j=1}^{r}\mathcal{O}_{H,S,T}^{\times}\otimes_{\mathbb{Z}}\mathbb{Z}[{\rm Gal}(H/F)]\to\lim_{\substack{\leftarrow\\
J
}
}\left(\bigotimes_{j=1}^{r}N_{v_{j}}^{(J)}\otimes_{\mathbb{Z}}\mathbb{Z}[{\rm Gal}(H/F)]\right).
\]
In Section \ref{sec:Conjecture}, we give a conjectural exact formula
for
\[
{\rm loc}(\epsilon_{H,S,T,V})\in\lim_{\substack{\leftarrow\\
J
}
}\left(\bigotimes_{j=1}^{r}N_{v_{j}}^{(J)}\otimes_{\mathbb{Z}}\mathbb{Z}[{\rm Gal}(H/F)]\right)
\]
by using $\hat{\Theta}_{S,T,V,J}$ (Conjecture \ref{Conj:main2LastSection}).
The conjectural construction of Gross-Stark units due to Dasgupta
\cite{MR2420508} and Dasgupta-Spiess \cite{DasguptaSpiess} can be
regarded as the $r=1$ case of our formula.

\emph{Acknowledgements.} The author would like to thank Nobuo Sato
and Takamichi Sano for helpful discussions.

\section*{Notation}

Throughout this paper, we use the same notation as in the introduction.
We fix a totally real field $F$, disjoint finite sets $(S,T)$ of
places of $F$ such that $S\supset S_{\infty}$, a subset $V=\{v_{1},\dots,v_{r}\}\subsetneq S$,
and an open subgroup $J=\prod_{v\in S}J_{v}\subset\prod_{v\in S}F_{v}^{\times}$.
We put $S_{f}=S\setminus S_{\infty}$. We denote by ${\bf Sub}(V)$
the category of subsets of $V$ whose morphism are inclusions. For
$U_{1}\subset U_{2}\subset V$, $r_{U_{1}}^{U_{2}}$ means a unique
element of ${\rm Hom}_{{\bf Sub}(V)^{{\rm op}}}(U_{2},U_{1})$. We
denote by $\mathbf{Mod}(F^{\times})$ the category of $\mathbb{Z}[F^{\times}]$-modules.
Unadorned tensor products are over $\mathbb{Z}$, except for Section
\ref{sec:Construction}. Unadorned tensor products are over $\mathbb{Z}[F^{\times}]$
in Section \ref{sec:Construction}. For a $\mathbb{Z}[F^{\times}]$-module
$M$, we denote $\ker(M\to M\otimes_{\mathbb{Z}[F^{\times}]}\mathbb{Z})$
by $IM$. Any map induced by the reciprocity map is denoted by ${\rm rec}.$
For $x\in F^{\times}$, we define ${\rm sgn}(x)\in\{\pm1\}$ by
\[
{\rm sgn}(x)=\prod_{v\in S_{\infty}}{\rm sgn}(\left|x\right|_{v}).
\]

\section{\label{sec:Construction}$(S,T,V,J)$-Shintani data and construction
of $Q(\mathcal{F},\varphi,x)$.}

In this section, if we omit the subscript of $\otimes$, it means
that the tensor product is over $\mathbb{Z}[F^{\times}]$. Let $R$
be a functor from ${\bf Sub}(V)^{{\rm op}}$ to ${\bf Mod}(F^{\times})$
defined by
\[
R(W)=\mathbb{Z}[N_{F}/\prod_{v\in V\setminus W}N_{v}].
\]

\begin{defn}
Let $\mathcal{F}$ be a functor from ${\bf Sub}(V)^{{\rm op}}$ to
${\bf Mod}(F^{\times})$, $\varphi$ a natural transform from $\mathcal{F}$
to $R$, and $x$ an element of $\mathcal{F}(V)/I\mathcal{F}(V)$.
We say that the triple $(\mathcal{F},\varphi,x)$ is a Shintani datum
if the following conditions hold.\end{defn}
\begin{enumerate}
\item $H_{i}(F^{\times},\mathcal{F}W)=0$ for all $i>0$ and $W\subset V$.
\item For all $W\subsetneq V$, ${\rm r}_{W}^{V}(\bar{x})\in I\mathcal{F}(W)$
where $\bar{x}\in\mathcal{F}(V)$ is a lift of $x$.\end{enumerate}
\begin{defn}
A Shintani datum $(\mathcal{F},\varphi,x)$ is an $(S,T,V,J)$-Shintani
datum if 
\[
{\rm rec}(\varphi(\bar{x}))=\Theta_{K,S,T}
\]
for all $K\in\Upsilon(J)$ where ${\rm rec}:\mathbb{Z}[N_{F}]\to\mathbb{Z}[{\rm Gal}(K/F)]$
is the map induced by the reciprocity map and $\bar{x}\in\mathcal{F}V$
is a lift of $x$.
\end{defn}
Let $(\mathcal{F},\varphi,x)$ be a Shintani datum. In this section,
we define an element $Q(\mathcal{F},\varphi,x)\in D/I_{F^{\times}}D$.

\subsection{Definition of $\mathcal{F}(n)$.}

Let $\mathcal{F}$ be a functor from ${\bf Sub}(V)^{{\rm op}}$ to
${\bf Mod}(F^{\times})$. We put
\[
\mathcal{F}(k)=\bigoplus_{\substack{W\subset V\\
\#W=k
}
}\mathcal{F}W.
\]
Let us fix the ordering of $V$. For $W_{1},W_{2}\subset V$, we define
$c(W_{1},W_{2})\in\{0,1,-1\}$ by
\[
c(W_{1},W_{2})=\begin{cases}
(-1)^{i-1} & W_{2}=\{w_{1},\dots,w_{k}\}\text{ and \ensuremath{W_{1}=W_{2}\setminus\{w_{i}\}}}\\
0 & \text{otherwise}
\end{cases}
\]
where $w_{1}<\cdots<w_{k}$. We define a homomorphism $\partial:\mathcal{F}(k)\to\mathcal{F}(k-1)$
by
\[
\partial(x)=\sum_{\substack{U\subset W\\
c(U,W)\neq0
}
}c(U,W){\rm r}_{U}^{W}(x)\ \ \ (x\in\mathcal{F}W).
\]
We have $\partial\circ\partial=0$.

\subsection{Double complexes.}

Let $(\mathcal{F},\varphi,x)$ be a Shintani datum. Let 
\[
\cdots\to\mathcal{I}_{3}\to\mathcal{I}_{2}\to\mathcal{I}_{1}\to\mathcal{I}_{0}\to\mathbb{Z}\to0
\]
 be a free resolution of $\mathbb{Z}$ in the category of $\mathbb{Z}[F^{\times}]$-module.
Put $\mathcal{F}_{i,j}=\mathcal{I}_{i}\otimes\mathcal{F}(j)$ and
$R_{i,j}=\mathcal{I}_{i}\otimes R(j).$

Let us consider the two complexes
\[
\xymatrix{\vdots\ar[d] & \vdots\ar[d] &  & \vdots\ar[d]\\
\mathcal{F}_{2,r}\ar[r]\ar[d] & \mathcal{F}_{2,r-1}\ar[r]\ar[d] & \cdots\ar[r] & \mathcal{F}_{2,0}\ar[r]\ar[d] & 0\\
\mathcal{F}_{1,r}\ar[r]\ar[d] & \mathcal{F}_{1,r-1}\ar[r]\ar[d] & \cdots\ar[r] & \mathcal{F}_{1,0}\ar[r]\ar[d] & 0\\
\mathcal{F}_{0,r}\ar[r] & \mathcal{F}_{0,r-1}\ar[r] & \cdots\ar[r] & \mathcal{F}_{0,0}\ar[r] & 0
}
\]
and
\[
\xymatrix{\vdots\ar[d] & \vdots\ar[d] &  & \vdots\ar[d]\\
R_{2,r}\ar[r]\ar[d] & R_{2,r-1}\ar[r]\ar[d] & \cdots\ar[r] & R_{2,0}\ar[r]\ar[d] & 0\\
R_{1,r}\ar[r]\ar[d] & R_{1,r-1}\ar[r]\ar[d] & \cdots\ar[r] & R_{1,0}\ar[r]\ar[d] & 0\\
R_{0,r}\ar[r] & R_{0,r-1}\ar[r] & \cdots\ar[r] & R_{0,0}\ar[r] & 0.
}
\]
These complexes play an important role in our construction of $Q(\mathcal{F},\varphi,x)$.
We denote the vertical (resp. horizontal) arrows by $\partial_{v}$
(resp. $\partial_{h}$).
\begin{prop}
The vertical sequence
\[
\cdots\to\mathcal{F}_{2,k}\to\mathcal{F}_{1,k}\to\mathcal{F}_{0,k}
\]
is exact for all $k=0,1,\dots,r$.\end{prop}
\begin{proof}
The claim is equivalent to the statement that
\[
H_{i}(\cdots\to\mathcal{I}_{2}\otimes_{\mathbb{Z}[E_{V}]}\mathcal{F}W\to\mathcal{I}_{1}\otimes\mathcal{F}W\to\mathcal{I}_{0}\otimes\mathcal{F}W\to0)=0
\]
for all $W\subset V$ and $i\geq1$ . We have
\begin{align*}
 & H_{i}(\cdots\to\mathcal{I}_{2}\otimes_{\mathbb{Z}[E_{V}]}\mathcal{F}W\to\mathcal{I}_{1}\otimes\mathcal{F}W\to\mathcal{I}_{0}\otimes\mathcal{F}W\to0)\\
= & {\rm Tor}_{i}^{F^{\times}}(\mathcal{F}W,\mathbb{Z})\\
= & {\rm Tor}_{i}^{F^{\times}}(\mathbb{Z},\mathcal{F}W)\\
= & H_{i}(F^{\times},\mathcal{F}W)
\end{align*}
Thus the claim follows from the definition.\end{proof}
\begin{prop}
\label{prop:ExactHor}The horizontal sequence
\[
R_{k,r}\to R_{k,r-1}\to\cdots\to R_{k,0}\to0
\]
is exact for all $k\in\mathbb{Z}$.\end{prop}
\begin{proof}
It is enough to prove the exactness of $R(r)\to R(r-1)\to\cdots\to R(0)\to0$
since $\mathcal{I}_{k}$ is a free $\mathbb{Z}[F^{\times}]$-module.
Fix $0\leq m\leq r-1$. Let us prove the exactness of $R(m+1)\to R(m)\to R(m-1)$.
We define the subspaces of $R(m)$ by
\[
R^{(k)}(m)=\bigoplus_{\substack{W\subset V\\
\#W=m\\
\{v_{1},\dots,v_{k}\}\subset W
}
}R(W)\subset R(m)\ \ \ (k=0,\dots,r).
\]
For $k=0,\dots,r-1$, let $P(k)$ be the following statement:\\
If $a\in R^{(k)}(m)\cap\ker(\partial)$ then there exists $b\in R(m+1)$
such that $a-\partial(b)\in R^{(k+1)}(m)$.

Let us prove $P(k)$. Assume that $a\in R^{(k)}(m)\cap\ker(\partial)$.
Then $a$ can be written as
\[
a=\sum_{\substack{W\subset V\\
\#W=m
}
}a_{W}\ \ \ \ (a_{W}\in R(W)).
\]
Note that we have
\[
{\rm r}_{W\setminus\{v_{j}\}}^{W}a_{W}=0
\]
for all $j=1,\dots,k$ and $W\supset\{v_{1},\dots,v_{k}\}$. Put 
\[
a_{1}=\sum_{\substack{W\subset V\\
\#W=m\\
v_{k+1}\in W
}
}a_{W}
\]
 and
\[
a_{2}=\sum_{\substack{W\subset V\\
\#W=m\\
v_{k+1}\notin W
}
}a_{W}.
\]
Then we have $a=a_{1}+a_{2}$. For $W\subset V$ such that $\#W=m$
and $\{v_{1},\dots,v_{k+1}\}\subset W$, Let $b_{W}\in R(W)$ be any
element such that
\[
{\rm r}_{W\setminus\{v_{k+1}\}}^{W}b_{W}=C(W\setminus\{v_{k+1}\},W)a_{W\setminus\{v_{k+1}\}}
\]
and
\[
{\rm r}_{W\setminus\{v_{j}\}}^{W}b_{W}=0\ \ \ (j=1,\dots,k).
\]
Put 
\[
b=\sum_{\substack{\substack{W}
\subset V\\
\#W=m\\
\{v_{1},\dots,v_{k+1}\}\subset W
}
}b_{W}.
\]
Then we have
\[
\partial b-a_{2}\in R^{(k+1)}(m).
\]
Since $a_{1}\in R^{(k+1)}(m)$, we have $a-\partial b\in R^{(k+1)}(m)$.
Thus $P(k)$ is proved. Since $R_{m}^{(0)}=R_{m}$ and $R_{m}^{(r)}=\{0\}$,
the claim is proved.
\end{proof}

\subsection{Definition of $Q(\mathcal{F},\varphi,x)$}
\begin{defn}
Let $(\mathcal{F},\varphi,x)$ be a Shintani datum. Take $a_{j}\in\mathcal{F}_{j,r-j}$
for each $j=0,1,\dots,r$ and $b_{j}\in R_{j,r-j+1}$ for each $j=1,\dots,r+1$.
We say that $(a_{0},\dots,a_{r},b_{1},\dots,b_{r+1})$ is a compatible
system for $(\mathcal{F},\varphi,x)$ if
\begin{itemize}
\item $a_{0}\in\mathcal{F}(V)$ is a lift of $x\in\mathcal{F}(V)/\left(I_{F^{\times}}\otimes\mathcal{F}(V)\right)$, 
\item $\partial_{v}(a_{j+1})=\partial_{h}(a_{j})\ \ \ \ (j=0,1,\dots,r-1)$,
\item and $\partial_{h}(b_{j})=\varphi(a_{j})-\partial_{v}(b_{j+1})\ \ \ \ (j=1,2\dots,r).$
\end{itemize}
\end{defn}
\begin{prop}
Let $(\mathcal{F},\varphi,x)$ be a Shintani datum. There exists a
compatible system for $(\mathcal{F},\varphi,x)$.\end{prop}
\begin{proof}
The existence of $a_{1}$ follows from the definition of Shintani
data. For $k=2,\dots,r$, the existence of $a_{k}$ follows from the
exactness of the vertical sequence. For $k=r,r-1,\dots,1$, the existence
of $b_{k}$ follows from Proposition \ref{prop:ExactHor}.\end{proof}
\begin{prop}
Let $(\mathcal{F},\varphi,x)$ be a Shintani datum. If $(a_{0},\dots,a_{r},b_{1},\dots,b_{r+1})$
is a compatible system for $(\mathcal{F},\varphi,x)$, then
\[
\varphi(a_{0})-\partial_{v}(b_{1})\in D.
\]
\end{prop}
\begin{proof}
Note that the ideal $D=\prod_{v\in V}I_{N_{v},N_{F}}$ is equal to
\[
\{x\in R(V)\mid{\rm r}_{V\setminus\{v\}}^{V}(x)=0\ \text{for all }v\in V\}.
\]
Thus the proposition follows from
\begin{align*}
\partial_{h}(\varphi(a_{0})-\partial_{v}(b_{1})) & =0.
\end{align*}
\end{proof}
\begin{prop}
$\varphi(a_{0})-\partial_{v}(b_{1})\pmod{I_{F^{\times}}D}$ does not
depend on the choice of a compatible system $(a_{0},\dots,a_{r},b_{1},\dots,b_{r+1})$.\end{prop}
\begin{proof}
Let $(a_{0}',\dots,a_{r}',b_{1}',\dots,b_{r}')$ be another compatible
system. Then there exists $c_{j}\in\mathcal{F}_{j,r+1-j}$ for $j=1,\dots,r+1$
and $d_{j}\in R_{j,r+2-j}$ for $j=2,\dots,r+2$ such that
\begin{align*}
a_{0}' & =a_{0}+\partial_{v}(c_{1})\\
a_{j}' & =a_{j}+\partial_{v}(c_{j+1})+\partial_{h}(c_{j})\ \ \ (j=1,\dots,r)\\
b_{j}' & =b_{j}+\varphi(c_{j})+\partial_{h}(d_{j})-\partial_{v}(d_{j+1})\ \ \ (j=2,\dots,r).
\end{align*}
Put
\[
e=\varphi(c_{1})-b_{1}'+b_{1}+\partial_{v}(d_{2})\in R_{1,r}.
\]
Then we have $e\in\mathcal{I}_{1}\otimes_{\mathbb{Z}[F^{\times}]}D$
since $\partial_{h}(e)=0$. We have
\[
\varphi(a_{0}')-\partial_{v}(b_{1}')-\varphi(a_{0})+\partial_{v}(b_{1})=\partial_{v}(e)\in I_{F^{\times}}D.
\]
Thus the claim is proved.\end{proof}
\begin{defn}
Let $(\mathcal{F},\varphi,x)$ be a Shintani datum. We define $Q(\mathcal{F},\varphi,x)\in D/\left(I_{F^{\times}}D\right)$
by
\[
Q(\mathcal{F},\varphi,x)=\varphi(a_{0})-\partial_{v}(b_{1})
\]
where $(a_{0},\dots,a_{r},b_{1},\dots,b_{r+1})$ is a compatible system
for $(\mathcal{F},\varphi,x)$.
\end{defn}
It is obvious that $Q(\mathcal{F},\varphi,x)$ does not depend on
the choice of a free resolution $\cdots\to\mathcal{I}_{1}\to\mathcal{I}_{0}\to\mathbb{Z}\to0$.
\begin{prop}
\label{prop:ChangeOfFunctorInQ}Let $\mathcal{F}$ and $\mathcal{G}$
be functors from ${\bf Sub}(V)^{{\rm op}}$ to ${\bf Mod}(F^{\times})$.
Let $\phi:\mathcal{F}\to\mathcal{G}$ and $\psi:\mathcal{G}\to R$
be natural transformations. Let $x$ be an element of $\mathcal{F}(V)/\left(I_{F^{\times}}\otimes\mathcal{F}(V)\right)$.
If $(\mathcal{F},\psi\circ\phi,x)$ and $(\mathcal{G},\psi,\phi_{V}(x))$
are Shintani data then 
\[
Q(\mathcal{F},\psi\circ\phi,x)=Q(\mathcal{G},\psi,\phi_{V}(x)).
\]
\end{prop}
\begin{proof}
If $(a_{0},\dots,a_{r},b_{1},\dots,b_{r+1})$ is a compatible system
for $(\mathcal{F},\psi\circ\phi,x)$ then $(\phi(a_{0}),\dots,\phi(a_{r}),b_{1},\dots,b_{r+1})$
is a compatible system for $(\mathcal{G},\psi,\phi_{V}(x))$. Thus
the claim holds. \end{proof}
\begin{thm}
\label{thm:GrossFromShintani}\textup{ If an $(S,T,V,J)$-Shintani
datum exists then
\[
\Theta_{E,S,T}\in\prod_{v\in V}I_{G_{v},G}
\]
for all $E\in\Upsilon(J)$.}\end{thm}
\begin{proof}
Let $(\mathcal{F},\varphi,x)$ be an $(S,T,V)$-Shintani datum and
$(a_{0},\dots,a_{r},b_{1},\dots,b_{r+1})$ a compatible system for
$(\mathcal{F},\varphi,x)$. Since $(\mathcal{F},\varphi,x)$ is an
$(S,T,V)$-Shintani datum, we have
\begin{equation}
{\rm rec}(\varphi(a_{0}))=\Theta_{E,S,T}.\label{eq:p1}
\end{equation}
Since $\varphi(a_{0})-\partial_{v}(b_{1})\in\prod_{v\in V}I_{N_{v},N_{F}}$
and $\partial_{v}(b_{1})\in I_{F^{\times}}$, we have 
\begin{align}
{\rm rec}(\varphi(a_{0})) & ={\rm rec}(\varphi(a_{0})-\partial_{v}(b_{1}))\in\prod_{v\in V}I_{G_{v},G}.\label{eq:p2}
\end{align}
The theorem follows from (\ref{eq:p1}) and (\ref{eq:p2}).
\end{proof}

\section{\label{sec:ShintaniDatumBLT}The construction of an $(S,T,V,J)$-Shintani
datum $(B,\mathcal{L},\vartheta)$.}

In this section, we construct an $(S,T,V,J)$-Shintani datum $(B,\mathcal{L},\vartheta)$
under the Condition (C1). We assume (C1) throughout this section.
We write $T'$ for the set of primes of $F$ which have a same residue
characteristic as some prime in $T$. We write $\mathsf{I}_{F}$ for
the group of fractional ideals of $F$, $\mathsf{I}_{(S)}$ for the
group of fractional ideals coprime to $S$, $\mathsf{I}_{(T)}$ for
the group of fractional ideals coprime to $T'$, $\mathsf{I}_{(S,T)}$
for the group of fractional ideals coprime to $S$ and $T'$, $\mathsf{P}_{(T)}$
for the group of principal fractional ideals coprime to $T'$, $F_{(S)}$
for the group of elements of $F^{\times}$ coprime to $S$, $F_{(T)}$
for the group of elements of $F^{\times}$ coprime to $T'$. For a
finite set $W$ of places of $F$, we put $N_{W}=\prod_{v\in W}N_{v}$.
Put $N^{S}=\prod'_{v\notin S}N_{v}=\prod'_{v\notin S}(F_{v}^{\times}/O_{v}^{\times})$.
Note that we have
\[
N_{F}=N_{S}\oplus N^{S}.
\]
We write $i_{S}:F^{\times}\to N_{F}$ for the natural composite map
$F^{\times}\to N_{S}\subset N_{F}$, and $i^{S}:F^{\times}\to N_{F}$
for the natural composite map $F^{\times}\to N^{S}\subset N_{F}$.
Note that $i_{S}+i^{S}$ is equal to a diagonal map $F^{\times}\to N_{F}$.
Let $\iota:\mathsf{I}_{F}\to N^{S}\subset N_{F}$ be the unique homomorphism
such that $\iota((x))=i^{S}(x)$ for all $x\in F^{\times}$. For $W\subset V$,
put $\bar{W}=W+(S\setminus V)$. For $W\subset S_{\infty}$, put $X_{W}=\prod_{v\in W}\left(F_{v}^{\times}/\mathbb{R}_{+}\right)$.
Put $E=\mathcal{O}_{F}^{\times}$. We denote by $E_{+}$ the group
of totally positive units of $F$. We put $n=[F:\mathbb{Q}]$.

\subsection{Definition of $B$.}

For a non-empty set $X$, let $C_{k}(X)$ denote a free $\mathbb{Z}$-module
generated by the formal symbols 
\[
(x_{1},\dots,x_{k})\ \ \ (x_{1},\dots,x_{k}\in X).
\]
There is an exact sequence
\[
\cdots\to C_{3}(X)\xrightarrow{d_{3}}C_{2}(X)\xrightarrow{d_{2}}C_{1}(X)\xrightarrow{d_{1}}C_{0}(X)\simeq\mathbb{Z}\xrightarrow{d_{0}}0
\]
where
\[
d_{m}((x_{1},\dots,x_{m}))=\sum_{i=1}^{m}(-1)^{i-1}(x_{1},\dots,\widehat{x_{i}},\dots,x_{m}).
\]
We put $\mathcal{C}(X)=C_{n}(X)/\ker(C_{n}(X)\xrightarrow{d_{n}}C_{n-1}(X)).$

If a group $A$ acts on $X$, we define the action of $A$ to $C_{k}(X)$
by
\[
a(x_{1},\dots,x_{k})=(ax_{1},\dots,ax_{k}).
\]

If $A$ acts on $X$ freely, there is a canonical isomorphism
\begin{equation}
H_{k-1}(A,\mathbb{Z})\simeq d_{k}^{-1}(I_{A}(C_{k-1}(X)))/(\ker\partial_{k}+I_{A}(C_{k}(X)))\label{eq:grouphomology}
\end{equation}
from the definition of the group homology.

For $W\subset S_{\infty}$ and $g\in X_{W}$, put 
\[
F_{g}=F_{(T)}\cap i^{-1}(g)\ \ \ \ (i:F^{\times}\to X_{W}\text{ is a diagonal map}).
\]
For $W\subset S_{\infty}$ and $k\geq0$, we define a $\mathbb{Z}[F_{(T)}]$-module
$C_{k,W}$ by the following way. The underlying $\mathbb{Z}$-module
of $C_{k,W}$ is 
\[
\bigoplus_{g\in X_{W}}C_{k}(F_{g}).
\]
We write an element of $C_{k,W}$ as 
\[
\sum_{g\in X_{W}}\Lambda_{g}\llbracket g\rrbracket\ \ \ \ \ (\Lambda_{g}\in C_{k}(F_{g})).
\]
We define the action of $F_{(T)}$ to $C_{k,W}$ by
\[
[x](\Lambda\llbracket g\rrbracket)={\rm sgn}(x)\cdot(x\Lambda)\llbracket xg\rrbracket\ \ \ \ (x\in F_{(T)}).
\]
For $W\subset S_{\infty}$, we define a $\mathbb{Z}[F_{(T)}]$-module
$\mathscr{K}_{W}$ by
\[
C_{n,W}/\ker(C_{n,W}\xrightarrow{d_{n}}C_{n-1,W}).
\]

For $W\subset S_{f}$, we denote by $\mathscr{A}_{W}\subset\mathbb{Z}[\mathsf{I}_{F}]$
the $\mathbb{Z}[\mathsf{I}_{(T)}]$-submodule generated by
\[
\prod_{\mathfrak{p}\in W}(1-[\mathfrak{p}])\prod_{\mathfrak{p}\in T}(1-N(\mathfrak{p})[\mathfrak{p}]).
\]

For $W\subset V$, we define a $\mathbb{Z}[F_{(T)}]$-module $\mathscr{M}_{W}$
by the following way. The underlying $\mathbb{Z}$-module of $\mathscr{M}_{W}$
is
\[
\mathscr{K}_{\bar{W}\cap S_{\infty}}\otimes\mathscr{A}_{\bar{W}\cap S_{f}}.
\]
The action of $F_{(T)}$ to $\mathcal{M}_{W}$ is defined by
\[
x(f\otimes g)=xf\otimes xg\ \ \ \ ((x,f,g)\in F_{(T)}\times\mathscr{K}_{\bar{W}\cap S_{\infty}}\times\mathscr{A}_{\bar{W}\cap S_{\infty}}).
\]

\begin{defn}
We define a functor $B:{\bf Sub}(V)^{{\rm op}}\to{\bf Mod}(\mathbb{Z}[F^{\times}])$
by
\[
B(W)=\left(\mathscr{M}_{W}\otimes\mathbb{Z}[N^{S}]\otimes\mathbb{Z}[F^{\times}]\right)/\mathscr{N}
\]
where $\mathscr{N}$ is an $\mathbb{Z}$-module spanned by
\[
\{x^{-1}a\otimes i^{S}(x)b\otimes c-a\otimes b\otimes xc\mid x\in F_{(T)},a\in\mathscr{M}_{W},b\in\mathbb{Z}[N^{S}],c\in\mathbb{Z}[F^{\times}]\}.
\]
We define the action of $F^{\times}$ to $B(W)$ by $x(a\otimes b\otimes c)=(a\otimes b\otimes xc)$.
\end{defn}

\subsection{Definition of functor $\beta$.}

For linearly independent vectors $x_{1},\dots,x_{m}\in F$ over $\mathbb{Q}$,
we define the rational cone $C(x_{1},\dots,x_{m})\subset F^{\times}$
by
\[
C(x_{1},\cdots,x_{m})=\{\sum_{i=1}^{m}t_{i}x_{i}\mid t_{i}\in\mathbb{Q}_{>0}\text{ for \ensuremath{i=1,\dots,m}}\}.
\]
For $W\subset S_{\infty}$, we denote by $\mathcal{K}_{W}$ the subgroup
of ${\rm Map}(F,\mathbb{Z})$ spanned by
\[
\{\bm{1}_{C(x_{1},\dots,x_{m})}\mid m>0,\ g\in X_{W},\ x_{1},\dots,x_{m}\in F_{g}\}.
\]
For $\gamma=\sum_{j}n_{j}\mathfrak{b}_{j}\in\mathbb{Z}[\mathsf{I}_{F}]$,
we put
\[
\bm{1}_{\gamma}=\sum_{j}n_{j}\bm{1}_{\mathfrak{b}_{j}}\in{\rm Map}(F,\mathbb{Z}).
\]
For $W\subset S_{f}$, we denote by $\mathcal{A}_{W}$ the subgroup
of ${\rm Map}(F,\mathbb{Z})$ spanned by $\{\bm{1}_{\gamma}\mid\gamma\in\mathscr{A}_{W}\}$.
For $W\subset V$, we denote by $\mathcal{M}_{W}$ the subgroup of
${\rm Map}(F^{\times},\mathbb{Z})$ spanned by
\[
\{f\cdot g\mid(f,g)\in\mathcal{K}_{\bar{W}\cap S_{\infty}}\times\mathcal{A}_{\bar{W}\cap S_{f}}\}.
\]

\begin{defn}
We define a functor $\beta:{\bf Sub}(V)^{{\rm op}}\to{\bf Mod}(\mathbb{Z}[F^{\times}])$
by
\[
\beta(W)=\left(\mathcal{M}_{W}\otimes\mathbb{Z}[N^{S}]\otimes\mathbb{Z}[F^{\times}]\right)/\mathcal{N}
\]
where $\mathcal{N}$ is an $\mathbb{Z}$-module spanned by
\[
\{x^{-1}a\otimes i^{S}(x)b\otimes c-a\otimes b\otimes xc\mid x\in F_{(T)},a\in\mathcal{M}_{W},b\in\mathbb{Z}[N^{S}],c\in\mathbb{Z}[F^{\times}]\}
\]
where
\[
p:N_{F}=N_{S}\oplus N^{S}\to N^{S}
\]
is the projection. We define the action of $F^{\times}$ to $\beta(W)$
by $x(a\otimes b\otimes c)=(a\otimes b\otimes xc)$.
\end{defn}

\subsection{Definition of natural transformation $\xi_{v,e}:B\to\beta$}

We use the Shintani cocycle defined in \cite{MR3351752}. Let us recall
the construction of it. Let us fix a determinant map ${\rm det}:F^{n}\to\mathbb{Q}.$
In this subsection, we fix an infinite place $v\in S_{\infty}\setminus V$
and $e\in\{\pm1\}$. Let $\rho$ be an embedding of $F$ into $\mathbb{R}$
corresponding to $v$. Let $w$ be a unique element of $F\otimes_{\mathbb{Q}}\mathbb{R}$
such that
\[
\rho'(w)=\begin{cases}
e & \rho'=\rho\\
0 & \rho'\neq\rho
\end{cases}
\]
for all the embeddings $\rho':F\to\mathbb{R}$. For linearly independent
vectors $x_{1},\dots,x_{m}\in F^{\times}$, we define the cone $C_{\infty}(x_{1},\dots,x_{m})\subset F\otimes_{\mathbb{Q}}\mathbb{R}$
by
\[
C_{\infty}(x_{1},\dots,x_{m})=\{\sum_{i=1}^{m}t_{i}x_{i}\mid t_{i}\in\mathbb{R}_{>0}\text{ for \ensuremath{i=1,\dots,m}}\}.
\]
The Shintani cocycle $\Xi_{v,e}\in{\rm Hom}_{\mathbb{Z}[F^{\times}]}(C_{n}(\rho^{-1}(\mathbb{R}_{>0}))\oplus C_{n}(\rho^{-1}(\mathbb{R}_{<0})),{\rm Map}(F^{\times},\mathbb{Z}))$
is defined as $\Xi_{v,e}(\Lambda(x_{1},\dots,x_{n}))=0$ if $x_{1},\dots,x_{n}\in F$
are linearly dependent over $\mathbb{Q}$ and as 
\begin{align*}
\Xi_{v,e}((x_{1},\dots,x_{n}))(y)= & {\rm sgn}({\rm det}(x_{1},\dots,x_{n}))\times\\
 & \ \ \begin{cases}
\lim_{\epsilon\to0+}1_{C_{\infty}(x_{1},\dots,x_{n})}(y+\epsilon w) & x_{1},\dots,x_{n}\text{ are in }\rho^{-1}(\mathbb{R}_{>0})\\
\lim_{\epsilon\to0+}1_{C_{\infty}(x_{1},\dots,x_{n})}(y-\epsilon w) & x_{1},\dots,x_{n}\text{ are in }\rho^{-1}(\mathbb{R}_{<0})
\end{cases}
\end{align*}
otherwise. The following theorem is proved in \cite{MR3351752}.
\begin{thm}
\label{thm:cocycle-relation}$\Xi_{v,e}$ satisfies the following
cocycle relation:
\[
\Xi_{v,e}(d_{n+1}\Lambda(x_{1},\dots,x_{n+1}))=0.
\]

\end{thm}
For $W\subset V$, we define $\phi_{W,v,e}:\mathscr{M}_{W}\to\mathcal{M}_{W}$
by
\[
\phi_{W,v,e}(\Lambda\otimes\gamma)=\Xi_{v,e}(j(\Lambda))\cdot\bm{1}_{\gamma}\ \ \ \ ((\Lambda,\gamma)\in\mathscr{K}_{\bar{W}\cap S_{\infty}}\times\mathscr{A}_{\bar{W}\cap S_{f}})
\]
where $j:\mathscr{K}_{\bar{W}\cap S_{\infty}}\to\mathcal{C}(F^{\times})$
is a map defined by
\[
j(\sum_{g}D_{g}\llbracket g\rrbracket)=\sum_{g}D_{g}.
\]
Note that $j$ is \emph{not }a $\mathbb{Z}[F_{(T)}]$-homomorphism
since 
\[
j([x]\Lambda)={\rm sgn}(x)[x]j(\Lambda)\ \ \ \ ((x,\Lambda)\in\mathscr{K}_{\bar{W}\cap S_{\infty}}\times F_{(T)}).
\]
\emph{ }
\begin{defn}
We define a natural transformation $\xi_{v,e}$ from $B$ to $\beta$
by
\[
\xi_{v,e}(W)(a\otimes b\otimes c)=\phi_{W,v,e}(a)\otimes b\otimes c.
\]

\end{defn}

\subsection{Definitions of natural transformations $\psi:\beta\to R$ and $\mathcal{L}_{v,e}:B\to R$}

For $f\in\mathcal{M}_{W}$, $\bm{s}=(s_{v})_{v\in\bar{W}\cap S_{\infty}}\in\mathbb{C}^{\#(\bar{W}\cap S_{\infty})}$
and $y\in N_{\bar{W}}$, define the Dirichlet series $\zeta_{f,y}(\bm{s})$
by
\[
\zeta_{f,y}(\bm{s})=\sum_{x\in i^{-1}(y)}f(x)\prod_{v\in\bar{W}_{\infty}}\left|x\right|_{v}^{-s_{v}}
\]
where $i:F^{\times}\to N_{\bar{W}}$ is the diagonal map. The Dirichlet
series $\zeta_{f,y}(\bm{s})$ converges absolutely if $\Re s_{v}>1$
for all $v\in\bar{W}\cap S_{\infty}$. The following statement can
be proved by the almost same way as that used in \cite[Section 6.1]{MR2420508}. 
\begin{prop}
The Dirichlet series $\zeta_{f,y}(\bm{s})$ has an analytic continuation
to a meromorphic function on the hole $\mathbb{C}^{\overline{W}_{\infty}}$,
and is holomorphic at $\bm{s}=\bm{0}$. Furthermore $\zeta_{f,y}(\bm{0})\in\mathbb{Z}$.
\end{prop}
The conditions concerning $T$ in Condition (C1) are used only for
this proposition. Since $\#\{y\in N_{\bar{W}}\mid\zeta_{f,y}(\bm{0})\neq0\}<\infty$,
the sum 
\[
\zeta_{f}(\bm{s})=\sum_{y\in N_{\bar{W}}}\zeta_{f,y}(\bm{s})[y^{-1}]\in\mathbb{Z}[N_{\bar{W}}]
\]
is well-defined. Put $\zeta_{f}=\zeta_{f}(\bm{0})\in\mathbb{Z}[N_{\bar{W}}]$.
We define the natural transformation $\psi$ from $\beta$ to $R$
by
\[
\psi(W)(f\otimes a\otimes b)=i_{1}(\zeta_{f})i_{2}(a)i_{3}(b)\ \ \ ((f,a,b)\in\mathcal{M}_{W}\times\mathbb{Z}[N^{S}]\times\mathbb{Z}[F^{\times}])
\]
where $i_{1}:\mathbb{Z}[N_{\bar{W}}]\to\mathbb{Z}[N_{F}/\prod_{v\in V\setminus W}N_{v}]$,
$i_{2}:\mathbb{Z}[N^{S}]\to\mathbb{Z}[N_{F}/\prod_{v\in V\setminus W}N_{v}]$
and $i_{3}:\mathbb{Z}[F^{\times}]\to\mathbb{Z}[N_{F}/\prod_{v\in V\setminus W}N_{v}]$
are natural map. This map is well-defined because $i_{1}(\zeta_{xf})=i_{S}(x)^{-1}i_{1}(\zeta_{f})$,
$i_{2}(xa)=i^{S}(x)i_{2}(a)$, and $i_{3}(xb)=xi_{3}(b)$ for $x\in F_{(T)}$.
We define a natural transformation $\mathcal{L}_{v,e}$ from $B$
to $R$ by $\mathcal{L}_{v,e}=\psi\circ\xi_{v,e}$.

\subsection{Definition of $\vartheta$.}

Put $X_{\infty}=X_{S_{\infty}}$. We denote by $i_{\infty}:F^{\times}\to X_{\infty}$
the natural diagonal map. Put $g_{0}=i_{\infty}(1)\in X_{\infty}$.
For a subgroup $A\subset F^{\times}$, put
\begin{align*}
Z_{A} & =\{D\in\mathcal{C}(F^{\times})\mid d_{n}(D)\in I_{A,A}\otimes_{\mathbb{Z}[A]}C_{n-1}(F^{\times})\}\\
U_{A} & =I_{A,A}\otimes_{\mathbb{Z}[A]}\mathcal{C}(F^{\times}).
\end{align*}
From a definition of group homology, there is a canonical isomorphism
\[
Z_{A}/U_{A}\simeq H_{n-1}(A,\mathbb{Z}).
\]
Put 
\[
\gamma_{A}:Z_{A}\to Z_{A}/U_{A}\simeq H_{n-1}(E_{+},\mathbb{Z}).
\]
 For a subgroup $A\subset F^{\times}$ and $c\in H_{n-1}(A,\mathbb{Z})$,
we put
\[
Y(c)=\{D\in Z_{A}\mid\gamma_{A}(D)=c\}.
\]
The following theorem is proved in \cite{MR3198753}.
\begin{thm}
\label{thm:SignedFD}There exists a canonical generator $\eta$ of
$H_{n-1}(E_{+},\mathbb{Z})$ such that
\[
\sum_{\epsilon\in E_{+}}\Xi_{v,e}(D)(x\epsilon)=\begin{cases}
1 & i_{\infty}(x)=g_{0}\\
0 & {\rm otherwise}
\end{cases}
\]
for all $D\in\mathcal{C}(F_{g_{0}})\cap Y(\eta)$ and $x\in F^{\times}$.
\end{thm}
We put $\delta_{S,T}=\prod_{\mathfrak{p}\in S_{f}}(1-[\mathfrak{p}])\prod_{\mathfrak{p}\in T}(1-N(\mathfrak{p})[\mathfrak{p}])$.
We denote by $C_{F}^{+}$ the narrow class group of $F$. 
\begin{defn}
For a narrow ideal class $C\in C_{F}^{+}$, define $\vartheta_{c}\in B(V)/IB(V)$
by
\[
\vartheta_{C}=\left(D\llbracket g_{0}\rrbracket\otimes\mathfrak{a}^{-1}\delta_{S,T}\right)\otimes\iota(\mathfrak{a})\otimes1
\]
where $\mathfrak{a}\in C\cap\mathsf{I}_{(T)}$ and $D\in\mathcal{C}(F_{g_{0}})\cap Y(\eta)$.
This definition does not depend on the choice of $\mathfrak{a}$ and
$D$.
\end{defn}

\begin{defn}
We define $\vartheta\in B(V)/IB(V)$ by
\end{defn}
\[
\vartheta=\sum_{C\in C_{F}^{+}}\vartheta_{C}.
\]

\subsection{Proof that $(B,\mathcal{L},\vartheta)$ is a Shintani datum.}
\begin{prop}
\label{prop:rVanish}Let $a\in B(V)$ be a lift of $\vartheta$. For
all $v\in V$, ${\rm r}_{V\setminus\{v\}}^{V}(a)\in IB(V\setminus\{v\})$.\end{prop}
\begin{proof}
Fix $v\in V$ and $D\in\mathcal{C}(F_{g_{0}})\cap Y(\eta)$. Let $f:B(V)/IB(V)\to B(V\setminus\{v\})/IB(V\setminus\{v\})$
be a homomorphism induced by ${\rm r}_{V\setminus\{v\}}^{V}$. What
we want to prove is that $f(\vartheta)=0$. We will consider the following
three cases:
\begin{enumerate}
\item $v$ is finite,
\item $v$ is infinite and there exists $\epsilon\in E$ such that $(\epsilon)_{v}<0$
and $(\epsilon)_{w}>0$ for all $w\in S_{\infty}-\{v\}$,
\item $v$ is infinite and there exists no such a $\epsilon\in E$.
\end{enumerate}
In the case (1), put $\mathfrak{p}=v$. Put $\delta_{S',T}=\prod_{\mathfrak{q}\in S_{f}\setminus\{v\}}(1-[\mathfrak{q}])\prod_{\mathfrak{q}\in T}(1-N(\mathfrak{q})[\mathfrak{q}])$.
For $C\in C_{F}^{+}$, we put
\[
\vartheta_{C}^{'}=\left(D\llbracket g_{0}\rrbracket\otimes\mathfrak{a}^{-1}\delta_{S',T}\right)\otimes\iota(\mathfrak{a})\otimes1\in B(V\setminus\{v\})/IB(V\setminus\{v\})
\]
where $\mathfrak{a}\in C\cap\mathsf{I}_{(T)}$. Then $\vartheta_{C}'$
does not depend on the choice of $\mathfrak{a}$. Since
\begin{align*}
f(\vartheta_{C}) & =\left(D\llbracket g_{0}\rrbracket\otimes\mathfrak{a}^{-1}\delta_{S,T}\right)\otimes\iota(\mathfrak{a})\otimes1\\
 & =\left(D\llbracket g_{0}\rrbracket\otimes(\mathfrak{a}^{-1}-\mathfrak{pa}^{-1})\delta_{S,T}\right)\otimes\iota(\mathfrak{a})\otimes1\\
 & =\vartheta_{C}^{'}-\vartheta_{\mathfrak{p}^{-1}C}',
\end{align*}
we have
\[
f(\vartheta)=\sum_{C\in C_{F}^{+}}\vartheta_{C}'-\sum_{C\in C_{F}^{+}}\vartheta_{\mathfrak{p}^{-1}C}'=0.
\]

In the case (2), let $E'$ be a subgroup of $E$ spanned by $E_{+}\cup\{\epsilon\}$.
Note that $E'$ is a free abelian group. Let $g_{0}'\in X_{S_{\infty}\setminus\{v\}}$
be a projection of $g_{0}$. It is enough to prove that $D\llbracket g_{0}'\rrbracket\in I_{E',E'}\otimes_{\mathbb{Z}[E']}\left(\mathcal{C}(F_{g_{0}'})\llbracket g_{0}'\rrbracket\right)$.
There exists an element $D'$ of $\mathcal{C}(F_{g_{0}'})\cap Z_{E'}$
such that
\begin{equation}
\gamma_{E'}(D)=\gamma_{E'}(2D').\label{eq:d0}
\end{equation}
Put $D''=(1+\epsilon)D'$. We have $D''\in Z_{E'}$ because 
\[
(1+[\epsilon])I_{E',E_{+}}\subset I_{E_{+},E_{+}}.
\]
Thus we have
\begin{equation}
D-D''\in Z_{E+}.\label{eq:d1}
\end{equation}
From (\ref{eq:d0}), we have
\begin{equation}
\gamma_{E'}(D-D'')=0.\label{eq:d2}
\end{equation}
From (\ref{eq:d1}) and (\ref{eq:d2}), we have $D-D''\in U_{E_{+}}$.
Thus we have
\begin{equation}
(D-D'')\llbracket g_{0}'\rrbracket\in I_{E_{+},E_{+}}\otimes_{\mathbb{Z}[E_{+}]}\left(\mathcal{C}(F_{g_{0}'})\llbracket g_{0}'\rrbracket\right).\label{eq:d3}
\end{equation}
Since $D''=(1+\epsilon)D'$, we have
\begin{align}
D''\llbracket g_{0}'\rrbracket & =([1]-[\epsilon])\left(D'\llbracket g_{0}'\rrbracket\right)\label{eq:d4}\\
 & \in I_{E',E'}\otimes_{\mathbb{Z}[E']}\left(\mathcal{C}(F_{g_{0}'})\llbracket g_{0}'\rrbracket\right).\nonumber 
\end{align}
From (\ref{eq:d3}) and (\ref{eq:d4}), we have
\[
D\llbracket g_{0}'\rrbracket\in I_{E',E'}\otimes_{\mathbb{Z}[E']}\left(\mathcal{C}(F_{g_{0}'})\llbracket g_{0}'\rrbracket\right).
\]
Thus $f(\vartheta)=0$.

In the case (3), let $x\in F_{(T)}$ be any element such that 
\[
{\rm sgn}((x)_{v'})=\begin{cases}
-1 & v'=v\\
1 & v'\neq v
\end{cases}
\]
for $v'\in S_{\infty}$. Then, for all $C\in C_{F}^{+}$, we have
\[
f(\vartheta_{C})+f(\vartheta_{(x)C})=0
\]
since
\begin{align*}
\vartheta_{(x)C} & =\left(D\llbracket g_{0}\rrbracket\otimes(x\mathfrak{a})^{-1}\delta_{S,T}\right)\otimes\iota(x\mathfrak{a})\otimes1\\
 & =-\left(xD\llbracket xg_{0}\rrbracket\otimes(\mathfrak{a})^{-1}\delta_{S,T}\right)\otimes\iota(\mathfrak{a})\otimes1\ \ \ \ (\mathfrak{a}\in C\cap\mathsf{I}_{(T)})
\end{align*}
and
\[
D-xD\in I_{E_{+},E_{+}}\otimes_{\mathbb{Z}[E_{+}]}\mathcal{C}(F_{g_{0}'})\ \ \ \ (g_{0}'\in X_{S_{\infty}\setminus\{v\}}\text{ is the projection of}\ g_{0}).
\]
Take $P\subset C_{F}^{+}$ such that $C_{F}^{+}=P\sqcup(x)P$. We
have 
\[
f(\vartheta)=\sum_{C\in P}\left(f(\vartheta_{C})+f(\vartheta_{(x)C})\right)=0.
\]

\end{proof}
In the proofs of Lemma \ref{lem:VanishingHomology} and Proposition
\ref{prop:ExactHor}, we use Shapiro's lemma many times.
\begin{lem}
\label{lem:VanishingHomology}For all $W\subsetneq S_{\infty}$ and
$i>0$, we have $H_{i}(E,\mathscr{K}_{W})=0$.\end{lem}
\begin{proof}
Fix $W\subsetneq S_{\infty}$ and $i>0$. From the definition, the
sequence
\[
\cdots\to C_{n+2,W}\to C_{n+1,W}\to C_{n,W}\to\mathscr{K}_{W}\to0
\]
is a free resolution of $\mathscr{K}_{W}$ in the category of $\mathbb{Z}[E]$-module.
Thus we have
\[
H_{i}(E,\mathscr{K}_{W})=H_{i+n-1}(E,C_{0,W}).
\]
The set $X_{W}$ can be written as
\[
X_{W}=g_{1}(E/E_{+})\sqcup\cdots\sqcup g_{s}(E/E_{+})
\]
where $g_{1},\dots,g_{s}\in X_{W}$. For $j=1,\dots,s$, put
\[
A_{j}=\mathbb{Z}\llbracket g_{j}\rrbracket\subset C_{0,W}
\]
and
\[
A_{j}'=\bigoplus_{g\in g_{j}(E/E_{+})}\mathbb{Z}\llbracket g\rrbracket\subset C_{0,W}.
\]
Put $U=\ker(E\to X_{W})$. Note that $U$ is a free abelian group
of rank $n-1$ since $W\neq S_{\infty}$. We have
\begin{align*}
H_{i+n-1}(E,C_{0,W}) & =\bigoplus_{j=1}^{s}H_{i+n-1}(E,A_{j}')\\
 & =\bigoplus_{j=1}^{s}H_{i+n-1}(E,A_{j}\otimes_{\mathbb{Z}[U]}\mathbb{Z}[E])\\
 & =\bigoplus_{j=1}^{s}H_{i+n-1}(U,A_{j}).
\end{align*}
Fix $1\leq j\leq s$. Note that the underlying $\mathbb{Z}$-module
of $A_{j}$ is isomorphic to $\mathbb{Z}$ and the action of $U$
to $A_{j}$ is defined by
\[
xa={\rm sgn}(x)a\ \ \ \ ((x,a)\in U\times A_{j}).
\]
Put $U'=\ker(U\xrightarrow{{\rm sgn}}\{\pm1\})$. If $U'=U$ then
\[
H_{i+n-1}(U,A_{j})\simeq\wedge^{i+n-1}U=0.
\]
If $U'\neq U$ then there exists a short exact sequence
\[
0\to\mathbb{Z}\to\mathbb{Z}\otimes_{\mathbb{Z}[U']}\mathbb{Z}[U]\to A_{j}\to0
\]
of $\mathbb{Z}[U]$-modules. This short exact sequence induces the
following exact sequence.
\[
\cdots\to H_{i+n-1}(U',\mathbb{Z})\to H_{i+n-1}(U,A_{j})\to H_{i+n-2}(U,\mathbb{Z})\to H_{i+n-2}(U',\mathbb{Z})\to\cdots.
\]
Here we have
\[
H_{i+n-1}(U',\mathbb{Z})=0
\]
and
\[
\ker\left(H_{i+n-2}(U,\mathbb{Z})\to H_{i+n-2}(U',\mathbb{Z})\right)=0.
\]
Thus $H_{i+n-1}(U,A_{j})=0$ and the claim is proved.\end{proof}
\begin{prop}
\label{prop:ExactVer}$H_{i}(F^{\times},B(W))=0$ for all $i>0$ and
$W\subset V$.\end{prop}
\begin{proof}
Fix $W\subset V$ and $i>0$. Let $X\subset N^{S}$ be the image of
$F_{(T)}$ under $\iota$. We put
\[
B'(W)=\mathscr{M}_{W}\otimes\mathbb{Z}[X]\otimes\mathbb{Z}[F^{\times}]/\mathscr{N}\subset B(W).
\]
Then $B(W)$ is isomorphic to a direct sum of copies of $B'(W)$.
So it is enough to prove that
\[
H_{i}(F^{\times},B'(W))=0
\]
for the proof of the proposition. Let $\overline{\mathscr{M}}_{W}$
be a $\mathbb{Z}[F^{\times}]$-module and $h:\overline{\mathscr{M}}_{W}\simeq\mathscr{M}_{W}$
a homomorphism of $\mathbb{Z}$-module such that
\[
x\cdot h(a)=h(x^{-1}a)\ \ \ (x,a)\in F^{\times}\times\mathcal{M}_{W}.
\]
Then $B'(W)$ is isomorphic to $\overline{\mathscr{M}}_{W}\otimes_{\mathbb{Z}[\mathcal{O}_{F,S}^{\times}]}\mathbb{Z}[F^{\times}]$.
Therefore we have
\[
H_{i}(F^{\times},B'(W))=H_{i}(\mathcal{O}_{F,S}^{\times},\overline{\mathscr{M}}_{W}).
\]
Note that $H_{i}(\mathcal{O}_{F,S}^{\times},\overline{\mathscr{M}}_{W})=0$
is equivalent to $H_{i}(\mathcal{O}_{F,S}^{\times},\mathscr{M}_{W})=0$.
Since $\mathscr{M}_{W}$ is isomorphic to a direct sum of copies of
$\mathscr{K}_{\bar{W}\cap S_{\infty}}\otimes_{\mathbb{Z}[E]}\mathbb{Z}[\mathcal{O}_{F,S}^{\times}]$,
it is enough to prove that
\[
H_{i}(\mathcal{O}_{F,S}^{\times},\mathscr{K}_{\bar{W}\cap S_{\infty}}\otimes_{\mathbb{Z}[E]}\mathbb{Z}[\mathcal{O}_{F,S}^{\times}])=0.
\]
Since we have
\[
H_{i}(\mathcal{O}_{F,S}^{\times},\mathscr{K}_{\bar{W}\cap S_{\infty}}\otimes_{\mathbb{Z}[E]}\mathbb{Z}[\mathcal{O}_{F,S}^{\times}])=H_{i}(E,\mathscr{K}_{\bar{W}\cap S_{\infty}})
\]
and $\bar{W}\cap S_{\infty}\neq S_{\infty}$, the proposition follows
from Lemma \ref{lem:VanishingHomology}.
\end{proof}
From Proposition \ref{prop:rVanish} and \ref{prop:ExactVer}, we
obtain the following proposition.
\begin{prop}
\label{prop:preSD}The triple $(B,\mathcal{L},\vartheta)$ is a Shintani
datum.
\end{prop}

\subsection{Proof that $(B,\mathcal{L},\vartheta)$ is an $(S,T,V,J)$-Shintani
datum.}

Let $E\in\Upsilon(J)$. First, let us recall the definition of the
Stickelberger function $\Theta_{E,S,T}(s)$.
\begin{defn}
The Stickelberger function $\Theta_{E,S,T}(s)\in\mathbb{C}[{\rm Gal}(E/F)]$
is defined by
\[
\Theta_{E,S,T}(s)=\sum_{\chi\in{\rm Hom}({\rm Gal}(E/F),\mathbb{C}^{\times})}L_{S,T}(\chi,s)e_{\bar{\chi}}
\]
where $L_{S,T}(\chi,s)$ is an analytic continuation of
\[
\prod_{\mathfrak{p}\notin S}(1-\chi(\sigma_{\mathfrak{p}})N(\mathfrak{p})^{-s})\prod_{\mathfrak{p}\in T}(1-\chi(\sigma_{\mathfrak{p}})N(\mathfrak{p})^{1-s})^{-1}\ \ \ (\Re(s)>1)
\]
and
\[
e_{\bar{\chi}}=\frac{1}{\#{\rm Gal}(E/F)}\sum_{\sigma\in{\rm Gal}(E/F)}\chi(\sigma)[\sigma]\in\mathbb{C}[G].
\]

\end{defn}
For integer ideal $\mathfrak{a}$, we put
\[
c_{T}(\mathfrak{a})=\prod_{\substack{\mathfrak{p}\in T\\
\mathfrak{p}\subset\mathfrak{a}
}
}(1-N(\mathfrak{p})).
\]
Then the Stickelberger element can be written as
\[
\Theta_{E,S,T}(s)=\sum_{\substack{\mathfrak{a}\in\mathsf{I}_{(S)}\\
\mathfrak{a}\subset\mathcal{O}_{F}
}
}c_{T}(\mathfrak{a})N(\mathfrak{a})^{-s}[{\rm rec}(\iota(\mathfrak{a}))^{-1}].
\]
For $C\in C_{F}^{+}$, we put
\[
\Theta_{E,S,T}(s,C)=\sum_{\substack{\mathfrak{a}\in\mathsf{I}_{(S)}\cap C\\
\mathfrak{a}\subset\mathcal{O}_{F}
}
}c_{T}(\mathfrak{a})N(\mathfrak{a})^{-s}[{\rm rec}(\iota(\mathfrak{a}))^{-1}].
\]

\begin{lem}
\label{lem:ThetaC}We have 
\[
\Theta_{E,S,T}(0,C)={\rm rec}(\mathcal{L}_{v,e}(\theta_{C}))\in\mathbb{Z}[{\rm Gal}(E/F)]
\]
 where $\theta_{C}\in B(V)$ is a lift of $\vartheta_{C}$. \end{lem}
\begin{proof}
Fix $\mathfrak{a}\in C\cap\mathsf{I}_{(S,T)}$. Let $D$ be any element
of $\mathcal{C}(F_{g_{0}})\cap Y(\eta)$. We put
\[
\theta_{C}=\left(D\llbracket g_{0}\rrbracket\otimes\mathfrak{a}^{-1}\delta_{S,T}\right)\otimes\iota(\mathfrak{a})\otimes1\in B(V).
\]
Then $\theta_{C}$ is a lift of $\vartheta_{C}$. We denote by $F_{+}$
the set of totally positive elements. Since any ideal $\mathfrak{b}$
in $C$ can be written as $\mathfrak{b}=x\mathfrak{a}$ where $x\in F_{+}$,
we have 
\begin{align}
\Theta_{E,S,T}(s,C) & =\sum_{x\in(\mathfrak{a}^{-1}\cap F_{(S)}\cap F_{+})/E_{+}}c_{T}(x\mathfrak{a})N(x\mathfrak{a})^{-s}[{\rm rec}(\iota(x\mathfrak{a}))^{-1}]\nonumber \\
 & =N(\mathfrak{a})^{-s}[{\rm rec}(\iota(\mathfrak{a})]\sum_{x\in(\mathfrak{a}^{-1}\cap F_{(S)}\cap F_{+})/E_{+}}c_{T}(x\mathfrak{a})N((x))^{-s}[{\rm rec}(i^{S}(x))].\label{eq:e0}
\end{align}
From the definition, we have
\begin{equation}
\bm{1}_{\mathfrak{a}^{-1}\delta_{S,T}}(x)=\begin{cases}
c_{T}(x\mathfrak{a}) & x\in\mathfrak{a}^{-1}\cap F_{(S)}\\
0 & x\in F\setminus\left(\mathfrak{a}^{-1}\cap F_{(S)}\right).
\end{cases}\label{eq:e1}
\end{equation}
Since ${\rm rec}:\mathbb{A}_{F}^{\times}\to\mathbb{Z}[{\rm Gal}(E/F)]$
vanishes on $F^{\times}\subset\mathbb{A}_{F}^{\times}$, we have 
\begin{equation}
{\rm rec}(i^{S}(x))={\rm rec}(i_{S}(x))^{-1}\ \ \ (x\in F^{\times}).\label{eq:e2}
\end{equation}
From (\ref{eq:e1}), (\ref{eq:e2}) and Theorem \ref{thm:SignedFD},
we have 
\begin{align}
\sum_{x\in(\mathfrak{a}^{-1}\cap F_{(S)}\cap F_{+})/E_{+}}c_{T}(x\mathfrak{a})N((x))^{-s}[{\rm rec}(i^{S}(x))] & =\sum_{x\in\mathfrak{a}^{-1}\cap F_{(S)}}\Xi_{v,e}(D)c_{T}(x\mathfrak{a})N((x))^{-s}[{\rm rec}(i_{S}(x))^{-1}]\nonumber \\
 & =\sum_{x\in F^{\times}}\Xi_{v,e}(D)(x)\cdot\bm{1}_{\mathfrak{a}^{-1}\delta_{S,T}}(x)\cdot N((x))^{-s}[{\rm rec}(i_{S}(x))^{-1}]\nonumber \\
 & =\sum_{x\in F^{\times}}\phi_{V,v,e}(D\llbracket g_{0}\rrbracket\otimes\mathfrak{a}^{-1}\delta_{S,T})(x)\cdot N((x))^{-s}[{\rm rec}(i_{S}(x))^{-1}].\label{eq:e3}
\end{align}
Put
\[
f=\phi_{V,v,e}\left(D\llbracket g_{0}\rrbracket\otimes\mathfrak{a}^{-1}\delta_{S,T}\right).
\]
From the definition, we have
\[
\zeta_{f}((s_{v})_{v})=\sum_{x\in F^{\times}}f(x)\prod_{v\in S_{\infty}}\left|x\right|_{v}^{-s_{v}}[i_{S}(x)^{-1}].
\]
Thus
\begin{equation}
{\rm rec}(\zeta_{f}((s,\dots,s)))=\sum_{x\in F^{\times}}f(x)N((x))^{-s}[{\rm rec}(i_{S}(x))^{-1}].\label{eq:e4}
\end{equation}
From (\ref{eq:e0}), (\ref{eq:e3}) and (\ref{eq:e4}), we have 
\[
\Theta_{E,S,T}(s,C)=N(\mathfrak{a})^{-s}[{\rm rec}(\iota(\mathfrak{a})]\cdot{\rm rec}(\zeta_{f}(s,\dots,s)).
\]
Thus
\begin{equation}
\Theta_{E,S,T}(0,C)={\rm rec}(\zeta_{f}[\iota(\mathfrak{a})]).\label{eq:e5}
\end{equation}
From the definition, we have
\begin{align}
\mathcal{L}_{v,e}(\theta_{C}) & =\psi(\xi_{v,e}(\theta_{C}))\nonumber \\
 & =\psi(f\otimes\iota(\mathfrak{a})\otimes1)\nonumber \\
 & =\zeta_{f}[\iota(\mathfrak{a})].\label{eq:e6}
\end{align}
From (\ref{eq:e5}) and (\ref{eq:e6}), we have 
\[
\Theta_{E,S,T}(0,C)={\rm rec}(\mathcal{L}_{v,e}(\theta_{C})).
\]
\end{proof}
\begin{lem}
\label{lem:Theta}We have 
\[
\Theta_{E,S,T}={\rm rec}(\mathcal{L}_{v,e}(\theta))\in\mathbb{Z}[{\rm Gal}(E/F)]
\]
 where $\theta\in B(V)$ is a lift of $\vartheta$\textup{.}\end{lem}
\begin{proof}
The claim follows from Lemma \ref{lem:ThetaC}.
\end{proof}
From Proposition \ref{prop:preSD} and Lemma \ref{lem:Theta} we obtain
the following theorem.
\begin{thm}
The triple $(B,\mathcal{L},\vartheta)$ is an $(S,T,V,J)$-Shintani
datum.
\end{thm}
From this theorem and Theorem \ref{thm:GrossFromShintani}, we obtain
the following corollary, which was already proved in \cite{DasguptaSpiess}.
\begin{cor}
\label{cor:GrossFromC1}If $(S,T,V)$-satisfies (C1) then \textup{
\[
\Theta_{E,S,T}\in\prod_{v\in V}I_{G_{v},G}
\]
for all finite abelian extensions $E$ of $F$ unramified outside
$S$.} 
\end{cor}

\section{\label{sec:ZGmodules}A certain submodule of $(\bigotimes_{\mathbb{Z}}^{r}\mathcal{O}_{H,S,T}^{\times})\otimes_{\mathbb{Z}}\mathbb{Z}[{\rm Gal}(H/F)].$}

Let $G$ be a finite abelian group. For a $\mathbb{Z}[G]$-module
$N$, we write $\mathbb{Q}N$ for $\mathbb{Q}\otimes_{\mathbb{Z}}N$
and $N[G]$ for $N\otimes_{\mathbb{Z}}\mathbb{Z}[G]$.
\begin{defn}
A $\mathbb{Z}[G]$-lattice is a finitely generated $\mathbb{Z}[G]$-module
which is free as $\mathbb{Z}$-module.
\end{defn}
Let $M$ be a $\mathbb{Z}[G]$-lattice. Define the action of $G$
to $(\otimes_{\mathbb{Z}}^{r}M)[G]$ by
\[
\sigma(m_{1}\otimes\cdots\otimes m_{r}\otimes[\tau])=m_{1}\otimes\cdots\otimes m_{r}\otimes[\sigma\tau]
\]
where $m_{1},\dots,m_{r}\in M$ and $\tau,\sigma\in G$. For $\sigma\in G$,
define $c_{\sigma}\in{\rm End}_{\mathbb{Z}}((\otimes_{\mathbb{Z}}^{r}M)[G])$
by
\[
c_{\sigma}(m_{1}\otimes\cdots\otimes m_{r}\otimes[\tau])=m_{1}^{\sigma}\otimes m_{2}\otimes\cdots\otimes m_{r}\otimes[\tau].
\]
Let 
\[
\mathfrak{S}_{r}=\{f:\{1,\dots n\}\to\{1,\dots,n\}\mid\text{f is a bijection}\}
\]
be the symmetric group of degree $r$. There is a natural action of
$\mathfrak{S}_{r}$ to $(\otimes_{\mathbb{Z}}^{r}M)[G]$. We say that
$m\in(\otimes_{\mathbb{Z}}^{r}M)[G]$ is antisymmetric if $fm={\rm sgn}(f)m$
for all $f\in\mathfrak{S}_{r}$.
\begin{defn}
$(M,G)_{\star}^{r}$ is the submodule of $(\otimes_{\mathbb{Z}}^{r}M)[G]$
defined by
\[
(M,G)_{\star}^{r}=\{m\in(\otimes_{\mathbb{Z}}^{r}M)[G]:m\text{ is antisymmetric and }c_{\sigma}(m)=\sigma m\text{ for all }\sigma\in G\}.
\]

\end{defn}
For $\varphi_{1},\dots,\varphi_{r}\in{\rm Hom}_{\mathbb{Z}[G]}(M,\mathbb{Z}[G])$,
we define $\varphi_{1}\wedge\cdots\wedge\varphi_{r}\in{\rm Hom}_{\mathbb{Z}[G]}(\wedge_{\mathbb{Z}[G]}^{r}M,\mathbb{Z}[G])$
by
\[
(\varphi_{1}\wedge\cdots\wedge\varphi_{r})(m_{1}\wedge\cdots\wedge m_{r})=\det(\varphi_{i}(m_{j}))\ \ \ \ (m_{1},\dots,m_{r}\in M).
\]
 In \cite{MR1385509}, Rubin defined $\wedge_{0}^{r}M\subset\mathbb{Q}\wedge_{\mathbb{Z}[G]}^{r}M$
by
\[
\wedge_{0}^{r}M=\{m\in\mathbb{Q}\wedge_{\mathbb{Z}[G]}^{r}M:(\varphi_{1}\wedge\cdots\wedge\varphi_{r})(m)\in\mathbb{Z}[G]\ \text{for all }\varphi_{1},\dots,\varphi_{r}\in{\rm Hom}_{\mathbb{Z}[G]}(M,\mathbb{Z}[G])\}.
\]
Define $P\in{\rm Hom}_{\mathbb{Z}[G]}\left(\mathbb{Q}\wedge_{\mathbb{Z}[G]}^{r}M,\mathbb{Q}(\bigotimes_{\mathbb{Z}}^{r}M)[G]\right)$
by 
\[
P(m_{1}\wedge\cdots\wedge m_{r})=\sum_{f\in\mathfrak{S}_{r}}{\rm sgn}(f)\bigotimes_{j=1}^{r}(\sum_{\sigma\in G}m_{f(j)}^{\sigma^{-1}}[\sigma]).
\]

\begin{prop}
\label{prop:wedgeM_iso_starM-1}$\wedge_{0}^{r}M$ and $(M,G)_{\star}^{r}$
are isomorphic by $P$.\end{prop}
\begin{proof}
For $l\in{\rm Hom}_{\mathbb{Z}}(M,\mathbb{Z})$, we define $\hat{l}\in{\rm Hom}_{\mathbb{Z}[G]}(M,\mathbb{Z}[G])$
by
\[
\hat{l}(m)=\sum_{\sigma\in G}l(m^{\sigma^{-1}})[\sigma].
\]
Note that we have
\[
{\rm Hom}_{\mathbb{Z}[G]}(M,\mathbb{Z}[G])=\{\hat{l}\mid l\in{\rm Hom}_{\mathbb{Z}}(M,\mathbb{Z})\}.
\]
For $l_{1},\dots,l_{r}\in{\rm Hom}_{\mathbb{Z}}(M,\mathbb{Z})$, we
have
\[
(\hat{l}_{1}\wedge\cdots\wedge\hat{l}_{r})=(l_{1}\otimes\cdots\otimes l_{r}\otimes{\rm id})\circ P.
\]
Therefore, for $m\in\mathbb{Q}\wedge_{\mathbb{Z}[G]}^{r}M$, we have
\[
m\in\wedge_{0}^{r}M\Leftrightarrow P(m)\in(\bigotimes_{\mathbb{Z}}^{r}M)[G].
\]
Thus it is enough to prove that $P$ gives an isomorphism between
$\mathbb{Q}\wedge_{\mathbb{Z}[G]}^{r}M$ and $\mathbb{Q}(M,G)_{\star}^{r}$.
It is obvious that $P(\mathbb{Q}\wedge_{\mathbb{Z}[G]}^{r}M)\subset\mathbb{Q}(M,G)_{\star}^{r}$
from the definition. Define $\bar{Q}\in{\rm Hom}_{\mathbb{Z}}(\mathbb{Q}(\bigotimes_{j=1}^{r}M)[G],\mathbb{Q}\wedge_{\mathbb{Z}[G]}^{r}M)$
by 
\[
\bar{Q}(m_{1}\otimes\cdots\otimes m_{r}[\sigma])=\frac{1}{(\#G)^{r}r!}m_{1}\wedge\cdots\wedge m_{r}^{\sigma}.
\]
Let $Q\in{\rm Hom}_{\mathbb{Z}}(\mathbb{Q}(M,G)_{\star}^{r},\mathbb{Q}\wedge_{\mathbb{Z}[G]}^{r}M)$
be the restriction of $\bar{Q}$. Let us prove that $Q$ is the inverse
function of $P$. We have $Q\circ P={\rm id}$ since
\begin{align*}
Q(P(m_{1}\wedge\cdots\wedge m_{r})) & =Q(\sum_{f\in\mathfrak{S}_{r}}{\rm sgn}(f)\bigotimes_{j=1}^{r}(\sum_{\sigma\in G}m_{f(j)}^{\sigma^{-1}}[\sigma]))\\
 & =m_{1}\wedge\cdots\wedge m_{r}
\end{align*}
for $m_{1},\dots,m_{r}\in M$. Let us prove that $P\circ Q={\rm id}$.
We define 
\[
R\in{\rm End}_{\mathbb{Z}}(\mathbb{Q}(\bigotimes_{j=1}^{r}M)[G],\mathbb{Q}(\bigotimes_{j=1}^{r}M)[G])
\]
by
\[
R(m_{1}\otimes\cdots\otimes m_{r}\otimes[\tau])=\frac{1}{(\#G)^{r}r!}\sum_{f\in\mathfrak{S}_{r}}{\rm sgn}(f)\sum_{\sigma_{1},\dots,\sigma_{r}\in G}m_{f(1)}^{\sigma_{1}^{-1}}\otimes\cdots\otimes m_{f(r)}^{\sigma_{r}^{-1}}[\sigma_{1}\cdots\sigma_{r}\tau]).
\]
Then we have $P\circ\bar{Q}=R$. since $R(m)=m$ for all $m\in\mathbb{Q}(M,G)_{\star}^{r}$
, we have $P\circ Q={\rm id}$. Thus the claim is proved. 
\end{proof}
For $\chi\in\hat{G}$, we define $r_{S}(\chi)\in\mathbb{Z}_{\geq r}$
and $e_{\chi}\in\mathbb{C}[G]$ by
\[
r_{S}(\chi)=\begin{cases}
\#\{v\in S\mid\chi(G_{v})=1\} & \chi\neq1\\
\#S-1 & \chi=1
\end{cases}
\]
and
\[
e_{\chi}=\frac{1}{\#G}\sum_{\gamma\in G}\chi(\gamma)[\gamma^{-1}].
\]

Let $H$ be a finite abelian extension of $F$ unramified outside
$S$. We define $\Lambda_{H,S,T}\subset(\bigotimes_{\mathbb{Z}}^{r}\mathcal{O}_{H,S,T}^{\times})\otimes_{\mathbb{Z}}\mathbb{Z}[{\rm Gal}(H/F)]$
by
\begin{align*}
\Lambda_{H,S,T} & =\{\alpha\in(\mathcal{O}_{H,S,T}^{\times},G)_{\star}^{r}\mid e_{\chi}\alpha=0\text{ for all }\chi\in\hat{G}\text{ such that }r_{S}(\chi)>r\}
\end{align*}
where $G={\rm Gal}(H/F)$.

\section{\label{sec:Conjecture}The conjecture}

\subsection{Refinement of Gross's leading term conjecture}

We fix an abelian extension $H$ of $F$ and places $w_{1},\dots,w_{r}$
of $H$ lying above $v_{1},\dots,v_{r}$. We assume that $H\in\Upsilon(J)$
and that $v_{1},\dots,v_{r}$ split completely at $H/F$. We assume
the following conjecture.
\begin{conjecture}
[Rubin-Stark conjecture, equivalent to Conjecture B in \cite{MR1385509}]\label{Conj:RubinStark-1}Let
$R_{\infty}:(\bigotimes_{\mathbb{Z}}^{r}\mathcal{O}_{H,S,T}^{\times})\otimes_{\mathbb{Z}}\mathbb{Z}[{\rm Gal}(H/F)]\to\mathbb{R}[{\rm Gal}(H/F)]$
be a homomorphism defined by
\[
R_{\infty}(u_{1}\otimes\cdots\otimes u_{r}\otimes[\sigma])=\left(\prod_{j=1}^{r}-\log\left|u_{j}\right|_{v_{j}}\right)[\sigma].
\]
Then there exists a unique element $\epsilon_{H,S,T,V}\in\Lambda_{H,S,T}$
such that
\[
R_{\infty}(\epsilon_{H,S,T,V})=\lim_{s\to0}s^{-r}\Theta_{H,S,T}(s).
\]

\end{conjecture}
Our main conjecture is as follows.
\begin{conjecture}
\label{Conj:mainLastSection}Assume that $(S,T,V)$-satisfies the
condition (C1). Let $\epsilon_{H,S,T,V}$ be as in Conjecture \ref{Conj:RubinStark-1}.
Fix $v\in S_{\infty}\setminus V$ and $e\in\{\pm1\}$. Then the element
$\hat{R}_{H/F,S,T,J}(\epsilon_{H,S,T,V})\in D/I_{H}D$ is equal to
the projection of 
\[
Q(B,\mathcal{L}_{v,e},\vartheta)\in D/I_{F^{\times}}D.
\]
\end{conjecture}
\begin{rem}
We expect $Q(B,\mathcal{L}_{v,e},\vartheta)$ to be independent of
the choice of $v$ and e. We conjecture that\\
\begin{equation}
H_{i}(F^{\times},\beta(W))=0\ \ \ (i>0,W\subset V)\label{eq:co1}
\end{equation}
and that
\begin{equation}
\xi_{v,e}(\vartheta)=\xi_{v',e'}(\vartheta)\label{eq:co2}
\end{equation}
for all $v,v',e,e'$. If (\ref{eq:co1}) and (\ref{eq:co2}) are true
then $Q(B,\mathcal{L}_{v,e},\vartheta)$ does not depend on the choice
of $v$ and $e$ since $Q(B,\mathcal{L}_{v,e},\vartheta)=Q(\beta,\psi,\xi_{v,e}(\vartheta))$
from Proposition \ref{prop:ChangeOfFunctorInQ}.
\end{rem}

\begin{rem}
Assume that $(S,T,V)$-satisfies the condition (C1). If Conjecture
\ref{Conj:mainLastSection} is true then Conjecture \ref{Conjecute:Gross}
is true for all $K\in\Upsilon(J)$.
\end{rem}
Let us take the inverse limit $\xleftarrow[J]{}$ in Conjecture \ref{Conj:mainLastSection}
. We assume that $(S,T,V)$-satisfies the condition (C1). Fix $v\in S_{\infty}\setminus V$
and $e\in\{\pm1\}$. To avoid the confusion, we write $D^{(J)}$ for
$D$, $I_{F^{\times}}^{(D)}$ for $I_{F^{\times}}$ and $I_{H}^{(D)}$
for $I_{H}$. We put 
\[
\hat{\Theta}_{S,T,V,J}=Q(B,\mathcal{L}_{v,e},\vartheta)\in D^{(J)}/I_{F^{\times}}^{(J)}D^{(J)}
\]
and 
\[
\hat{\Theta}_{S,T,V}=\lim_{\substack{\leftarrow\\
J
}
}\hat{\Theta}_{S,T,V,J}\in\lim_{\substack{\leftarrow\\
J
}
}\left(D^{(J)}/I_{F^{\times}}^{(J)}D^{(J)}\right).
\]
Let 
\[
p_{H}:\lim_{\substack{\leftarrow\\
J
}
}\left(D^{(J)}/I_{F^{\times}}^{(J)}D^{(J)}\right)\to\lim_{\substack{\leftarrow\\
J
}
}\left(D^{(J)}/I_{H}^{(J)}D^{(J)}\right)
\]
be a natural projection map. Put 
\[
\hat{R}_{H/F,S,T}=\lim_{\substack{\leftarrow\\
J
}
}\hat{R}_{H/F,S,T,J}.
\]

\begin{conjecture}
\label{Con:mainLastSectionProj}Let $\epsilon_{H,S,T,V}$ be as in
Conjecture \ref{Conj:RubinStark-1}. The following equality holds.

\[
p_{H}(\hat{\Theta}_{S,T,V})=\hat{R}_{H/F,S,T}(\epsilon_{H,S,T,V}).
\]

\end{conjecture}
Conjecture \ref{Conj:mainLastSection} and Conjecture \ref{Con:mainLastSectionProj}
are equivalent.

\subsection{Conjectural construction of localized Rubin-Stark element}

Let us give an interpretation of Conjecture \ref{Con:mainLastSectionProj}
as a conjectural construction of localized Rubin-Stark element. We
assume that $(S,T,V)$-satisfies the condition (C1). Fix $v\in S_{\infty}\setminus V$
and $e\in\{\pm1\}$. We put $\hat{\Theta}_{S,T,V,J}=Q(B,\mathcal{L}_{v,e},\vartheta)$.
The regulator map $\hat{R}_{H/F,S,T}$ factor through ${\rm loc}$,
i.e., there exists a unique homomorphism
\[
\hat{\mathcal{R}}_{H/F,S,T}:\lim_{\substack{\leftarrow\\
J
}
}\left(\bigotimes_{j=1}^{r}N_{v_{j}}^{(J)}\otimes_{\mathbb{Z}}\mathbb{Z}[{\rm Gal}(H/F)]\right)\to\lim_{\substack{\leftarrow\\
J
}
}\left(D^{(J)}/I_{H}^{(J)}D^{(J)}\right)
\]
such that $\hat{R}_{H/F,S,T}=\hat{\mathcal{R}}_{H/F,S,T}\circ{\rm loc}$.
Define a homomorphism $\bar{c}_{J}$ from $\mathbb{Z}[N_{F}]$ to
$\bigotimes_{j=1}^{r}N_{v_{j}}^{(J)}\otimes_{\mathbb{Z}}\mathbb{Z}[{\rm Gal}(H/F)]$
by
\[
\bar{c}_{J}([x])=p_{1}(x)\otimes\cdots\otimes p_{r}(x)\otimes[{\rm rec}(x)]
\]
where $p_{j}:N_{F}^{(J)}\to N_{v_{j}}^{(J)}$ are natural projections.
The homomorphism $\bar{c}_{J}$ vanishes on $I_{H}^{(J)}D^{(J)}$.
Let 
\[
c_{J}:D^{(J)}/I_{H}^{(J)}D^{(J)}\to\bigotimes_{j=1}^{r}N_{v_{j}}^{(J)}\otimes_{\mathbb{Z}}\mathbb{Z}[{\rm Gal}(H/F)]
\]
be the map induced from $\bar{c}_{J}$. Put 
\[
c=\left(\lim_{\substack{\leftarrow\\
J
}
}c_{J}\right):\lim_{\substack{\leftarrow\\
J
}
}\left(D^{(J)}/I_{H}^{(J)}D^{(J)}\right)\to\lim_{\substack{\leftarrow\\
J
}
}\left(\bigotimes_{j=1}^{r}N_{v_{j}}^{(J)}\otimes_{\mathbb{Z}}\mathbb{Z}[{\rm Gal}(H/F)]\right).
\]
Then $c\circ\hat{\mathcal{R}}_{H/F,S,T}$ and $\hat{\mathcal{R}}_{H/F,S,T}\circ c$
are identities. 
\begin{defn}
We define the \emph{analytic Rubin-Stark element} 
\[
\epsilon_{H,S,T,V}^{{\rm an}}\in\lim_{\substack{\leftarrow\\
J
}
}\left(\bigotimes_{j=1}^{r}N_{v_{j}}^{(J)}\otimes_{\mathbb{Z}}\mathbb{Z}[{\rm Gal}(H/F)]\right)
\]
 by
\[
\epsilon_{H,S,T,V}^{{\rm an}}=c\circ p_{H}(\hat{\Theta}_{S,T,V}).
\]

\end{defn}
Since $\hat{\mathcal{R}}_{H/F,S,T}\circ c={\rm id}$, we have $p_{H}(\hat{\Theta}_{S,T,V})=\hat{\mathcal{R}}_{H/F,S,T}(\epsilon_{H,S,T,V}^{{\rm an}})$.
\begin{conjecture}
\label{Conj:main2LastSection}Let $\epsilon_{H,S,T,V}$ be as in Conjecture
\ref{Conj:RubinStark-1}. Then we have
\[
{\rm loc}(\epsilon_{H,S,T,V})=\epsilon_{H,S,T,V}^{{\rm an}}.
\]

\end{conjecture}
Conjecture \ref{Con:mainLastSectionProj} and Conjecture \ref{Conj:main2LastSection}
are equivalent.

\bibliographystyle{plain}
\bibliography{HigherRankDasgupta}

\end{document}